\documentclass[11pt]{amsart}

\usepackage{amsthm, amsfonts, amscd, amssymb,pinlabel,graphicx}
\usepackage[all]{xy}
\usepackage{hyperref}
\usepackage[usenames,dvipsnames]{xcolor}

\graphicspath{{Figures/}}

\def\TM+{T^*(\rr_+ \times M)}

\newcommand{\rr}{\ensuremath{\mathbb{R}}}

\theoremstyle{plain}
\newtheorem{thm}{Theorem}[section]

\newtheorem{lem}[thm]{Lemma}

\newtheorem{conj}[thm]{Conjecture}

\theoremstyle{definition}

\theoremstyle{remark}
\newtheorem{rem}[thm]{Remark}
\newtheorem{ex}[thm]{Example}

\numberwithin{equation}{section} 

\newcommand{\dfn}[1]{{\textbf {#1}}}

\newcommand{\leg}{\ensuremath{\Lambda}}

\begin{document}

\title{Positive Knots and Lagrangian Fillability}

\date{\today}

\author[K. Hayden]{Kyle Hayden} \address{Boston College,
Chestnut Hill, MA 02467} \email{kyle.hayden@bc.edu}  

\author[J. Sabloff]{Joshua M. Sabloff} \address{Haverford College,
Haverford, PA 19041} \email{jsabloff@haverford.edu} 

\begin{abstract} This paper explores the relationship between the existence of  an exact embedded Lagrangian filling for a Legendrian knot in the standard contact $\rr^3$ and the hierarchy of positive, strongly quasi-positive, and quasi-positive knots.  On one hand, results of Eliashberg and especially Boileau and Orevkov show that every Legendrian knot with an exact, embedded Lagrangian filling is quasi-positive.   On the other hand, we show that if a knot type is positive, then it has a Legendrian representative with an exact embedded Lagrangian filling. Further, we produce  examples that show that strong quasi-positivity and fillability are independent conditions.
\end{abstract}

\maketitle


\section{Introduction}
\label{sec:intro}

Properly embedded Lagrangian submanifolds of $B^4$ whose boundaries are Legendrian links in $S^3$, called \dfn{fillings} of the Legendrian links, are of interest in a variety of fields:  in smooth knot theory, Lagrangian fillings minimize the slice genus of a link \cite{chantraine}; in Legendrian knot theory, Lagrangian fillings induce augmentations of the Legendrian Contact Homology DGA \cite{ekholm:lagr-cob, ehk:leg-knot-lagr-cob, rizell:lifting}; and Lagrangian fillings can even be used to answer questions about complex algebraic curves \cite{polyfillability}. These considerations motivate the following question:
\begin{quote}
	\emph{Which smooth knot types have Legendrian representatives with Lagrangian fillings?}
\end{quote}
Such a smooth knot type is termed \dfn{fillable}. In analyzing this question, we work in the equivalent setting of Legendrian links in the standard contact $\rr^3$ and Lagrangian fillings in the symplectization $\rr \times \rr^3$. We further require the Lagrangian fillings to be exact, orientable, embedded, and collared, i.e. equal to $\rr \times \Lambda$  outside a compact set.

Initial progress on the question above indicates a close relationship to the hierarchy of positivity in smooth knot theory.  To describe the hierarchy, let $BP$ be the set of braid positive knot types, $P$ be the set of positive knot types, $SQP$ be the set of strongly quasi-positive knot types, and $QP$ be the set of quasi-positive knot types. The following relationships are well known (see \cite{hedden:pos, rudolph:survey, stoimenow:pos}, for example):
\begin{equation} \label{eq:hierarchy}
	BP \subsetneq P \subsetneq SQP \subsetneq QP.
\end{equation}

The first main result of this paper delineates a sufficient condition for a smooth knot type to be fillable:

\begin{thm} \label{thm:pos-filling}
	All positive knots are fillable.
\end{thm}

To make progress towards a necessary condition, we begin by noting that quasi-positivity is necessary for fillability.  To see this, first note that a result of Eliashberg \cite{yasha:lagr-cyl} can be extended to show that a Lagrangian filling of a Legendrian knot may be perturbed to a symplectic filling of a transverse knot (see \cite{polyfillability} for more details).  From there, use the following result of Boileau and Orevkov \cite{bo:qp}: if  a smooth knot type has a transverse representative with a symplectic filling, then it is quasi-positive.  

As we shall see in Section~\ref{sec:sqp-qp}, however, not all quasi-positive knots are fillable.  Further, the intermediate condition of strong quasi-positivity is independent of fillability:

\begin{thm} \label{thm:sqp-bad}
	There exists a fillable knot that is quasi-positive, but not strongly quasi-positive, and there exists non-fillable knot that is strongly quasi-positive.
\end{thm}

Since this theorem shows that strong quasi-positivity is not relevant to fillability, we must seek an alternative condition to characterize fillable knots.  Based on the results above and a survey of quasi-positive knots up to ten crossings,\footnote{We used \cite{baader:slice-gordian} for a list of quasi-positive knots,  Morrison's code in the \texttt{KnotTheory} package \cite{knottheory} and Morton and Short's \texttt{C++} program \cite{morton-short:homfly-program, morton-short:homfly-article} to compute the HOMFLY polynomials, \texttt{Gridlink} \cite{gridlink} to find Legendrian representatives of quasi-positive knots, and our own constructions of Lagrangian fillings.} we make the following conjecture:

\begin{conj} \label{conj}
	A smooth knot type $K$ is fillable if and only if it is quasi-positive and the HOMFLY bound on the maximum Thurston-Bennequin number of $K$ is sharp.
\end{conj}

One implication of the conjecture is true:  if a smooth knot is fillable, then, as discussed above, it is quasi-positive.  Further, the Legendrian contact homology DGA of the fillable Legendrian representative has an augmentation \cite{ekholm:lagr-cob, rizell:lifting}, hence a graded ruling \cite{fuchs-ishk, rulings}, and hence the HOMFLY bound is sharp \cite{rutherford:kauffman}.  

The remainder of the paper is organized as follows:  after reviewing background on the various notions of positivity touched on above, on rulings of Legendrian knots, and on the construction of Lagrangian cobordisms via handle attachment in Section~\ref{sec:background}, we prove Theorem~\ref{thm:pos-filling} in Section~\ref{sec:pos-fillable} and Theorem~\ref{thm:sqp-bad} in Section~\ref{sec:sqp-qp}.

\subsection*{Acknowledgements} We thank Matt Hedden for several stimulating discussions about the material in this paper. The authors were 
partially supported by NSF grant DMS-0909273.

\section{Background}
\label{sec:background}

In this section, we review positivity and (strong) quasi-positivity of smooth knots, rulings of Legendrian knots, and various constructions of Lagrangian fillings.  We assume that the reader is familiar with basic notions of Legendrian knot theory as discussed, for example, in Etnyre's survey \cite{etnyre:knot-intro}.  In particular, we assume familiarity with front diagrams, the classical invariants, and the HOMFLY bound on the Thurston-Bennequin number.

\subsection{Notions of Positivity}
\label{ssec:gh}

The notion of positivity is simple to define: an oriented link is \dfn{positive} if it has a projection in which all crossings are positive. The jump to quasi-positivity requires the concept of positive bands in a braid, which generalize positive crossings. Denote the standard generators of the braid group by $\sigma_i$.  A positive band is a braid word of the form $w \sigma_i w^{-1}$, where $w$ is  any word in the braid group.  The more restricted notion of a positive embedded band, denoted $\sigma_{i,j}$, is a positive band of the form
\begin{equation*}\label{eq:sigmaindices} \sigma_{i,j} \equiv (\sigma_i  \cdots \sigma_{j-2}) \sigma_{j-1}(\sigma_i \cdots \sigma_{j-2})^{-1} , \qquad 1 \leq i < j \leq n. \end{equation*}
We define a (strongly) quasi-positive braid to be a product of positive (embedded) bands, and we say that an oriented link is \dfn{(strongly) quasipositive} if it is the closure of such a braid.

The relationship between these notions of positivity --- in particular, the inclusions displayed in Equation \eqref{eq:hierarchy} --- has been treated extensively in the literature; see \cite{hedden:pos, rudolph:survey, stoimenow:pos} for surveys.

\subsection{Rulings of Legendrian Links}
\label{ssec:rulings}

A ruling is a combinatorial object associated to the front diagram of a Legendrian link. Hereafter, assume that all Legendrian links have been perturbed so that the cusps and crossings in their front diagrams have distinct $x$-coordinates. We define a \textbf{ruling} of such a front to be a one-to-one correspondence between left and right cusps, together with pairs of paths (called \dfn{companions}) in the front joining corresponding cusps.  The companion paths must satisfy the following conditions:
\begin{enumerate}
\item paired paths intersect only at the cusps they join; and
\item unpaired paths intersect only at crossings.
\end{enumerate} 
In particular, the $x$-coordinate of each path is strictly monotonic and the paths cover the front. Further, since the companion paths  meet only at the cusps, each pair of companions bounds a \dfn{ruling disk} in the plane.

To refine the notion of a ruling, we refer to the paths incident to a crossing of the front diagram as \dfn{crossing paths}. At a crossing, either the two crossing paths pass through each other or one path lies entirely above the other. In the latter case, we say that the ruling is \dfn{switched} at the crossing.  We call a ruling \dfn{normal} if, at each switched crossing, the crossing paths have one of the configurations shown in Figure~\ref{fig:normality}. We can rephrase this definition in terms of ruling disks:  a ruling is normal if, near a switch, the interiors of the disks involved in the switch are either nested or disjoint.  If all switches occur at positive crossings, then we say that the ruling is \dfn{oriented}.\footnote{In the literature, such a ruling is also called \dfn{$2$-graded}.}

\begin{figure}
\includegraphics[width=.55\linewidth]{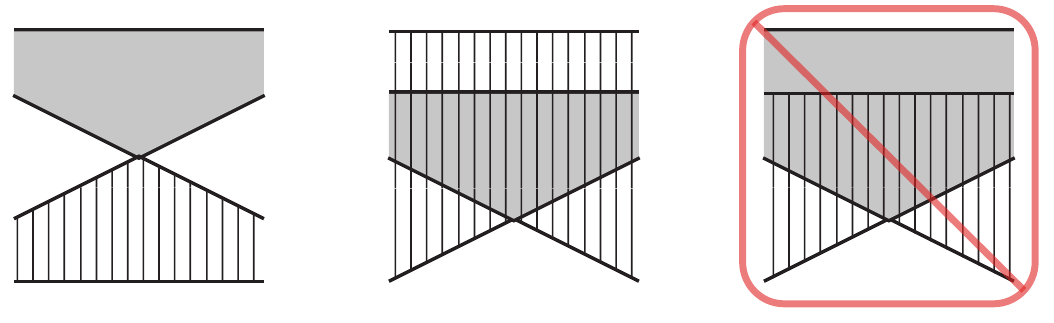}
\caption{The two leftmost configurations of crossing and companion strands are allowed in a normal ruling (as is the reflection of the  center configuration through a horizontal line), but the rightmost configuration is not allowed.}
\label{fig:normality}
\end{figure}

\begin{ex}
	The set of all oriented normal rulings of a Legendrian trefoil appear in Figure~\ref{fig:trefoil}.
\end{ex}

\begin{figure}
\begin{center}
	\includegraphics[width=.8\linewidth]{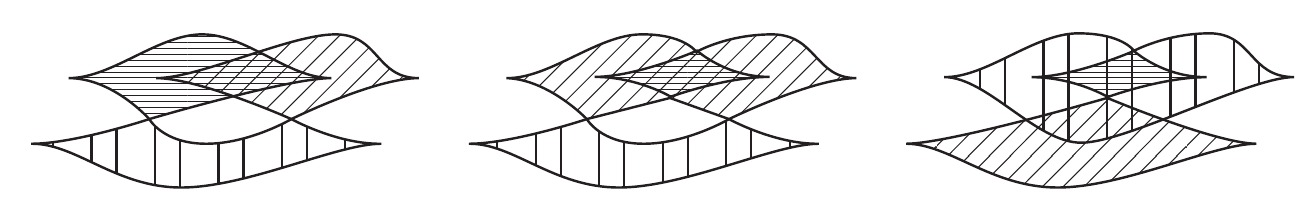}
\caption{The three normal rulings of a Legendrian trefoil.}
\label{fig:trefoil}
\end{center}
\end{figure}

The proof of Theorem~\ref{thm:pos-filling} relies on rulings that switch at every crossing. In such a ruling, ruling disks are either (globally) nested or disjoint, and hence such a ruling must be normal. This type of ruling is equivalent to an ``admissible 0-resolution'' of a front, as studied by Ng in \cite{lenny:khovanov}. A 0-resolution of a front is the diagram obtained by ``smoothing'' crossings as in Figure~\ref{fig:0-resolution}.  The aforementioned paper shows that alternating links have Legendrian representatives with normal --- but not necessarily oriented --- rulings that switch at each crossing. In Section ~\ref{sec:pos-fillable}, below, we will revisit a result of Tanaka which implies that positive links also admit \emph{oriented} normal rulings of this form.  

\begin{figure}
\includegraphics[width=.3 \linewidth]{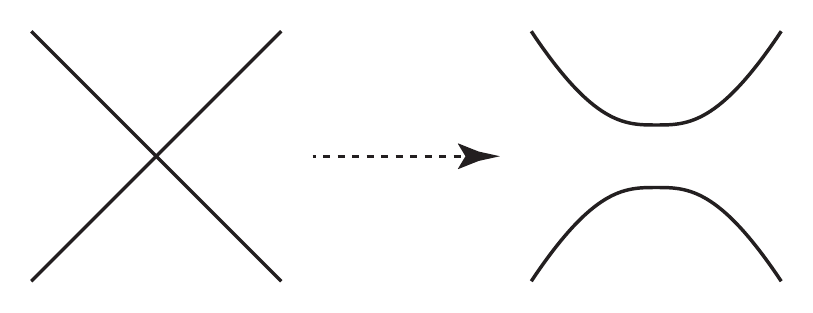}
\caption{Smoothing a crossing using a 0-resolution.}\label{fig:0-resolution}
\end{figure}

\subsection{Lagrangian Handle Attachment}
\label{ssec:cobord}

In this subsection, we describe constructions for Lagrangian fillings of Legendrian links in the standard contact $\rr^3$. A Lagrangian submanifold $L$ of the symplectization is a \textbf{Lagrangian cobordism} between the Legendrian submanifolds $\leg_\pm \subset \rr^3$ if there exists a pair of real numbers $T_- < T_+$ such that
\begin{align*}
L \cap \big((-\infty,T_-] \times \rr^3 \big) &= (-\infty,T_-] \times \leg_-,  \text{ and}\\
L \cap \big([T_+,\infty) \times \rr^3 \big) &= [T_+,\infty) \times \leg_+.
\end{align*}

As defined in the introduction, a Lagrangian filling of a Legendrian link is a cobordism from the empty set to the given link.  Note that stacking one Lagrangian cobordism on top of another results in a new Lagrangian cobordism.

 The following tool allows us to construct Lagrangian fillings with a prescribed Legendrian link as their boundary. It has been adapted from Theorem~4.2 of \cite{bst:construct}, though this result appears in slightly different forms in the work of Ekholm-Honda-K\'al\'man \cite{ehk:leg-knot-lagr-cob} and of Rizell \cite{rizell:surgery}.

\begin{thm}[\cite{bst:construct, chantraine, ehk:leg-knot-lagr-cob, rizell:surgery}]
\label{thm:construct}
If two Legendrian links $\leg_{-}$ and $\leg_{+}$ in the standard contact $\rr^{3}$ are related by any of the following three moves, then there exists an exact, embedded, orientable, and collared Lagrangian cobordism from $\leg_{-}$ to $\leg_{+}$.
\begin{description}
\item[Isotopy] $\leg_{-}$ and $\leg_{+}$ are Legendrian isotopic.
\item[$\boldsymbol{0}$-Handle] The front of $\leg_{+}$ is the same as that of $\leg_{-}$ except for the addition of a disjoint Legendrian unknot as in the left side of Figure~\ref{fig:handle}.
\item[$\boldsymbol{1}$-Handle] The fronts of $\leg_{\pm}$ are related as in the right side of Figure~\ref{fig:handle}.
\end{description}
\end{thm}

\begin{figure}
\includegraphics[width=.6\linewidth]{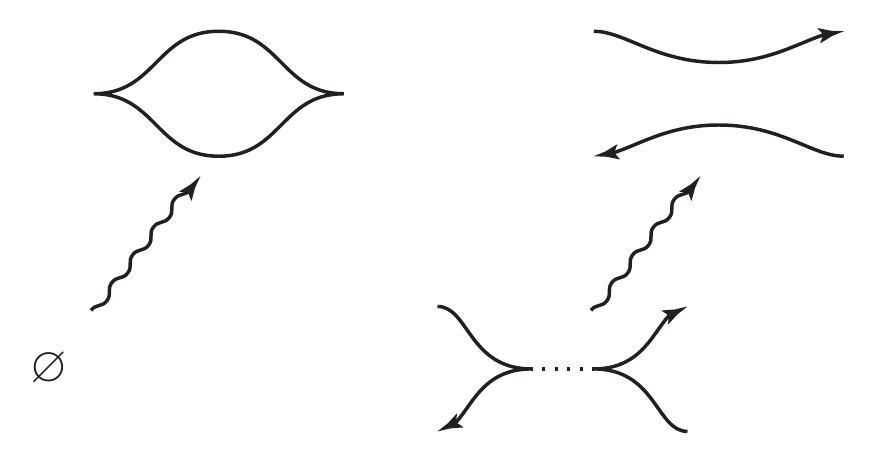} 
\caption{Diagram moves corresponding to attaching a 0-handle and an oriented 1-handle.}
\label{fig:handle}
\end{figure}

We will call a filling \dfn{decomposable} if it can be split into cobordisms arising as in Theorem~\ref{thm:construct}. Hereafter, Theorem~\ref{thm:construct} will be our primary means for producing links with Lagrangian fillings.

\section{Positive Knots are Fillable}
\label{sec:pos-fillable}

In this section, we prove Theorem~\ref{thm:pos-filling}. As mentioned above, the key tool in the proof is the existence of an oriented ruling in which all crossings are switched. The following lemma, essentially due to Tanaka \cite{tanaka:max-tb-pos}, shows that we may assume the existence of such a ruling for a positive link.

\begin{lem} \label{lem:pos-ruling}
Every positive link admits a Legendrian representative whose front diagram carries an oriented normal ruling in which all crossings are switched.
\end{lem}

\begin{proof}


In \cite{tanaka:max-tb-pos}, Tanaka extends an earlier result of Yokota \cite{yokota:pos-poly} to show that every positive link has a Legendrian representative $\leg$ with the following properties:
\begin{enumerate}
\item \label{it:crossing} Every crossing has both strands oriented to the left (in particular, every crossing is positive); and
\item \label{it:0-res} The 0-resolution of $\leg$ consists of disjoint (but possibly nested) closed curves, each of which contains exactly one left cusp and one right cusp.
\end{enumerate}
By (\ref{it:crossing}), the smoothing of each crossing in the 0-resolution respects the orientation of $\leg$. Thus, the strands which connect left and right cusps in the 0-resolution correspond to ruling paths in an oriented normal ruling that switches at each crossing.
\end{proof}

We now prove that all positive links are fillable.

\begin{figure}[t] 
	\center
	\labellist
		\small
		\pinlabel $1$-handles [b] at 762 1470
		\pinlabel Isotopy [b] at 762 600
		\endlabellist
		\includegraphics[width=.6\linewidth]{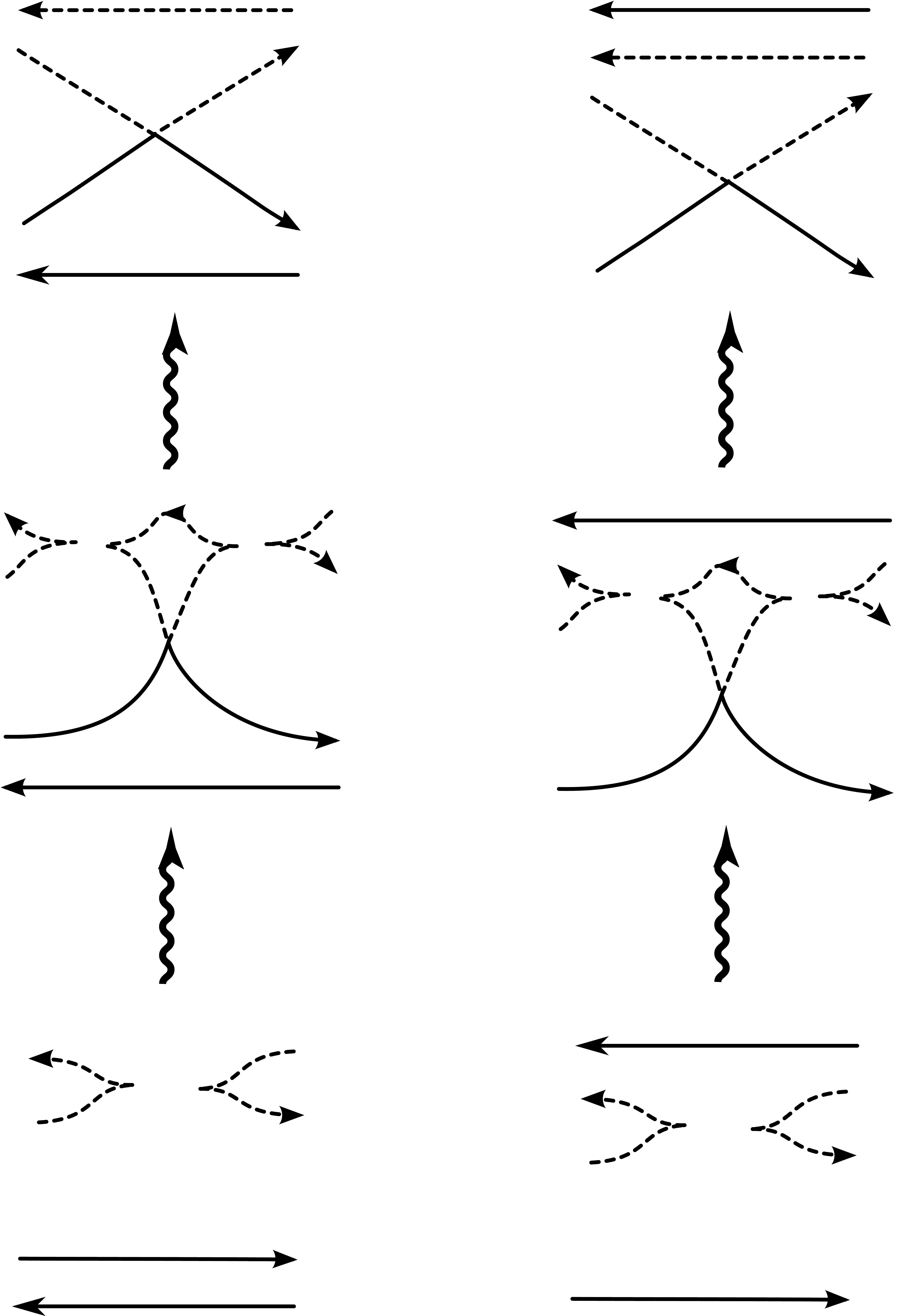}
\caption{At the top are diagrams of $\leg_+$ near the crossing constructed in the proof of Theorem~\ref{thm:pos-filling}.  Diagrams of $\leg_-$ appear at the bottom, and passing from bottom to top is a sequence of moves defining a  Lagrangian cobordism from $\leg_-$ up to $\leg_+$.} \label{fig:allowedcrossings}
\end{figure}

\begin{proof}[Proof of Theorem~\ref{thm:pos-filling}.]  By Lemma~\ref{lem:pos-ruling}, it suffices to consider a fixed front $\leg_+$ that admits an oriented normal ruling in which all crossings are switched.  The proof proceeds by induction on the number of crossings.  If $\leg_+$ has no crossings, then it must be a disjoint union of maximal Thurston-Bennequin unknots.  Such a link is fillable by the $0$-handle construction in Theorem~\ref{thm:construct}.

For the inductive step, suppose that every Legendrian whose front diagram has fewer crossings than $\leg_+$ and that admits an oriented normal ruling in which all crossings are switched is fillable.  We begin by showing that the ruling of $\leg_+$ must have a switch with a neighborhood equivalent to one of the topmost diagrams in Figure~\ref{fig:allowedcrossings} (up to an overall reversal of orientation or reflection through the horizontal). First, each switch occurs at a positive crossing, so the crossing paths must have the same horizontal direction. Second, companion paths are oriented in the opposite direction to the crossing paths, so the two companion strands must also have the same horizontal direction. Finally, because the ruling switches at each crossing, the ruling disks are either nested or disjoint. Any switch along  the boundary of an innermost ruling disk will have the desired form.

Next, we use the crossing found above to construct a Legendrian link $\leg_-$ with the following properties:
\begin{enumerate}
\item The fronts of $\leg_+$ and $\leg_-$ are identical outside of a neighborhood of the crossing;
\item Near the crossing, the front of $\leg_-$ is of one of the forms depicted at the bottom of Figure~\ref{fig:allowedcrossings};
\item $\leg_-$ has an oriented normal ruling with all crossings switched; and
\item There exists a Lagrangian cobordism from $\leg_-$ to $\leg_+$.
\end{enumerate}
Only the last condition needs some verification, which is carried out in Figure~\ref{fig:allowedcrossings}.

Clearly, the Legendrian $\leg_-$ satisfies the inductive hypothesis and has one less crossing than $\leg_+$.  It follows that $\leg_-$ is fillable, and hence, by condition (4), that $\leg_+$ is fillable as well.  The theorem follows. \end{proof}



\begin{rem}
	We say that an oriented ruling of a front has an \dfn{unlinked resolution} if the result of performing $0$-resolutions (see Figure \ref{fig:0-resolution}) at all switches of the ruling results in an unlink, all of whose components are maximal $tb$ unknots.  Using this language, we  see that the proof above actually gives a stronger result than Theorem~\ref{thm:pos-filling}: a smooth knot type is fillable if it has a Legendrian representative whose front has an oriented ruling with an unlinked resolution.
\end{rem}

\section{(Strong) Quasi-positivity and Fillability}
\label{sec:sqp-qp}

We end this paper by providing the examples necessary to prove Theorem~\ref{thm:sqp-bad}.  The first example proves the first part of the theorem, namely that there exists a fillable quasi-positive knot that is not strongly quasi-positive.

\begin{ex}
	As shown in Figure~\ref{fig:9-46}, the mirror of the $9_{46}$ knot is fillable, and hence quasi-positive.  On the other hand, since this knot is a non-trivial slice knot, its slice genus differs from its Seifert genus and hence it is not strongly quasi-positive  \cite[Prop.\ 2]{rudolph}.
\end{ex}

\begin{figure}
\begin{center}
\labellist
		\small
		\pinlabel $\emptyset$ [b] at 165 5
		\pinlabel $0$-handles [l] at 175 65
		\pinlabel Isotopy [l] at 175 185
		\pinlabel Isotopy [l] at 175 385
		\pinlabel $1$-handle [l] at 175 585
		\endlabellist
\includegraphics[width=.4\linewidth]{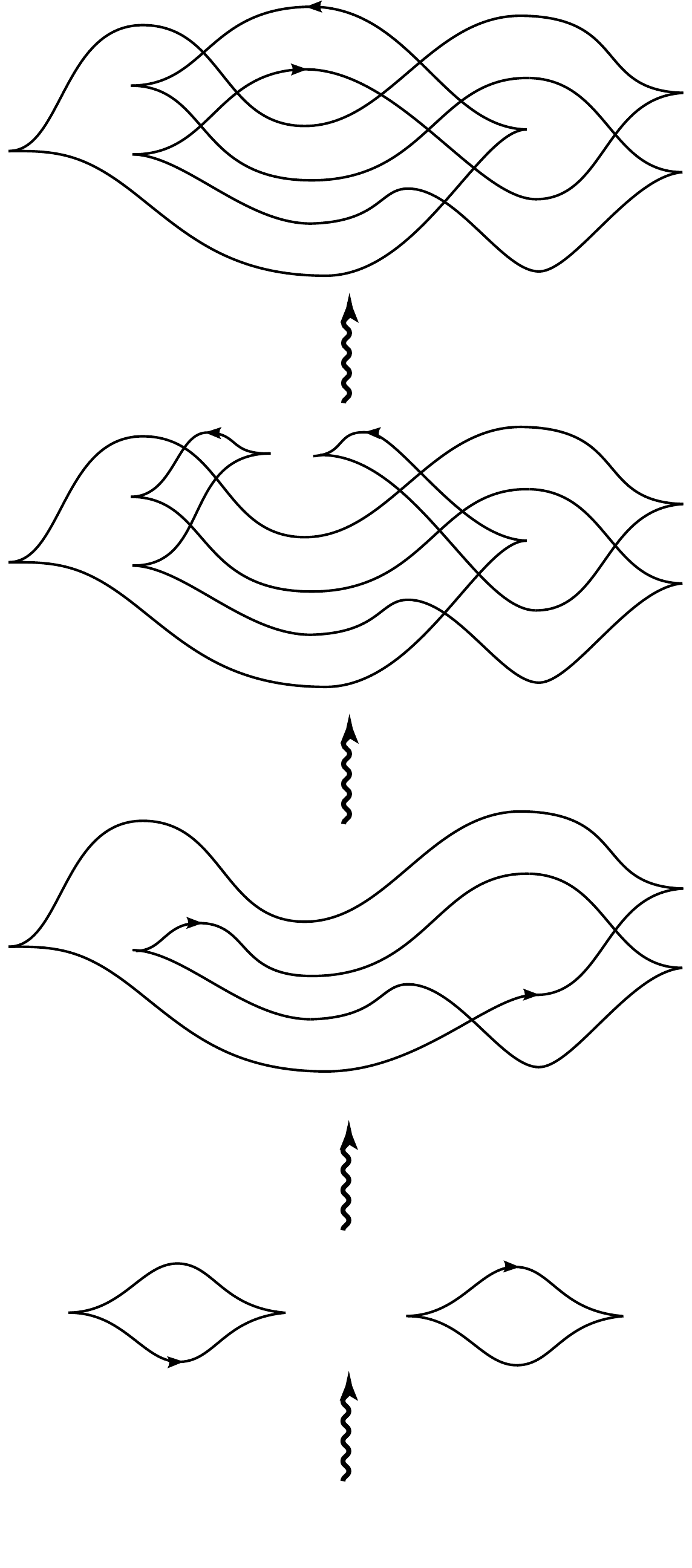}
\caption{The slice knot $m(9_{46})$ is fillable.}
\label{fig:9-46}
\end{center}
\end{figure}
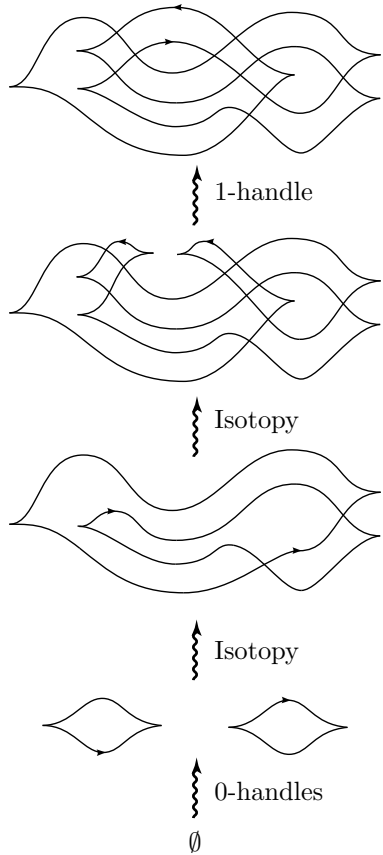

The example for the second part of the theorem --- that there exists a non-fillable strongly quasi-positive knot --- is somewhat more complicated.  It was originally brought to light by Stoimenow \cite[Example 3]{stoimenow:pos} in the context of disproving a conjecture of Fiedler on an upper bound for the minimal degree of the Jones polynomial of a (quasi-positive) knot in terms of a band representation.

\begin{ex}
	Let $K$ be the closure of the strongly quasi-positive $4$-braid
	$$\sigma_1^2  (\sigma_1 \sigma_2 \sigma_1^{-1}) \sigma_2 \sigma_1 \sigma_3 (\sigma_1 \sigma_2 \sigma_1^{-1}) \sigma_2 (\sigma_2 \sigma_3 \sigma_2^{-1}) (\sigma_1 \sigma_2 \sigma_1^{-1}) (\sigma_2 \sigma_3 \sigma_2^{-1}).$$
	
	Suppose, for the sake of contradiction, that $K$ has a fillable Legendrian representative $\Lambda$. Denote the maximum Euler characteristic of a smooth slicing surface for $K$ by $\chi_4(K)$.	We may easily compute from Rudolph's formula \cite[Section 3]{rudolph:qp-obstruction} that $\chi_4 (K) = -7$.  Since $\Lambda$ is fillable, Theorem~1.3 of \cite{chantraine} implies that $tb(\Lambda) = - \chi_4(K) = 7$ and that this is the maximal Thurston-Bennequin invariant for $K$, $\overline{tb}(K)$.
	
	On the other hand, we compute\footnote{Again, the computation was performed using Morrison's code in the \texttt{KnotTheory} package \cite{knottheory} and Morton and Short's \texttt{C++} program \cite{morton-short:homfly-program,morton-short:homfly-article}.} that the degree in the framing variable of the HOMFLY polynomial of $K$ is $-10$.  Thus, the HOMFLY bound on the Thurston-Bennequin number would be $\overline{tb}(K) \leq 9$.  Clearly, this bound is not sharp.   As argued after Conjecture~\ref{conj}, however, if $K$ were fillable, then the HOMFLY bound would be sharp.  We must conclude, then, that $K$ is not fillable.
\end{ex}


\bibliographystyle{amsplain} 
\bibliography{main}

\def\cprime{$'$} \def\polhk#1{\setbox0=\hbox{#1}{\ooalign{\hidewidth
  \lower1.5ex\hbox{`}\hidewidth\crcr\unhbox0}}}
\providecommand{\bysame}{\leavevmode\hbox to3em{\hrulefill}\thinspace}
\providecommand{\MR}{\relax\ifhmode\unskip\space\fi MR }
\providecommand{\MRhref}[2]{%
  \href{http://www.ams.org/mathscinet-getitem?mr=#1}{#2}
}
\providecommand{\href}[2]{#2}
\begin{thebibliography}{10}

\bibitem{knottheory}
\emph{{KnotTheory`}}, Mathematica Package available at
  \texttt{katlas.math.toronto.edu}, accessed June 2013.

\bibitem{baader:slice-gordian}
S.~Baader, \emph{Slice and {G}ordian numbers of track knots}, Osaka J. Math.
  \textbf{42} (2005), no.~1, 257--271.

\bibitem{bo:qp}
M.~Boileau and S.~Orevkov, \emph{Quasi-positivit\'e d'une courbe analytique
  dans une boule pseudo-convexe}, C. R. Acad. Sci. Paris S\'er. I Math.
  \textbf{332} (2001), no.~9, 825--830.

\bibitem{bst:construct}
F.~Bourgeois, J.~Sabloff, and L.~Traynor, \emph{{L}agrangian cobordisms via
  generating families with applications to {L}egendrian geography and botany},
  In preparation.

\bibitem{polyfillability}
C.~Cao, N.~Gallup, K.~Hayden, and J.~Sabloff, \emph{Topologically distinct
  {L}agrangian and symplectic fillings}, In preparation.

\bibitem{chantraine}
B.~Chantraine, \emph{On {L}agrangian concordance of {L}egendrian knots},
  Algebr. Geom. Topol. \textbf{10} (2010), 63--85.

\bibitem{gridlink}
M.~Culler, \emph{Gridlink: a tool for knot theorists}, Available at
  \texttt{www.math.uic.edu/\~{}culler/gridlink}, accessed June 2013.

\bibitem{ekholm:lagr-cob}
T.~Ekholm, \emph{Rational {SFT}, linearized {L}egendrian contact homology, and
  {L}agrangian {F}loer cohomology}, Perspectives in analysis, geometry, and
  topology, Progr. Math., vol. 296, Birkh\"auser/Springer, New York, 2012,
  pp.~109--145.

\bibitem{ehk:leg-knot-lagr-cob}
T.~Ekholm, K.~Honda, and T.~K\'alm\'an, \emph{{L}egendrian knots and exact
  {L}agrangian cobordisms}, Preprint available as arXiv:1212.1519, 2012.

\bibitem{yasha:lagr-cyl}
Ya. Eliashberg, \emph{Topology of {$2$}-knots in {${\bf R}\sp 4$} and
  symplectic geometry}, The Floer memorial volume, Progr. Math., vol. 133,
  Birkh\"auser, Basel, 1995, pp.~335--353.

\bibitem{etnyre:knot-intro}
J.~Etnyre, \emph{Legendrian and transversal knots}, Handbook of knot theory,
  Elsevier B. V., Amsterdam, 2005, pp.~105--185.

\bibitem{fuchs-ishk}
D.~Fuchs and T.~Ishkhanov, \emph{Invariants of {L}egendrian knots and
  decompositions of front diagrams}, Mosc. Math. J. \textbf{4} (2004), no.~3,
  707--717.

\bibitem{hedden:pos}
M.~Hedden, \emph{Notions of positivity and the {O}zsv\'ath-{S}zab\'o
  concordance invariant}, J. Knot Theory Ramifications \textbf{19} (2010),
  no.~5, 617--629.

\bibitem{morton-short:homfly-program}
H.~R. Morton and H.~B. Short, \emph{Homfly program}, Available at
  \texttt{http://www.liv.ac.uk/\~{}su14/knotprogs.html}, accessed June 2013.

\bibitem{morton-short:homfly-article}
\bysame, \emph{Calculating the {$2$}-variable polynomial for knots presented as
  closed braids}, J. Algorithms \textbf{11} (1990), no.~1, 117--131.

\bibitem{lenny:khovanov}
L.~Ng, \emph{{A Legendrian Thurston-Bennequin bound from Khovanov homology}},
  Algebr. Geom. Topol. \textbf{5} (2005), 1637--1653.

\bibitem{rizell:surgery}
G.~Rizell, \emph{{L}egendrian ambient surgery and {L}egendrian contact
  homology}, Preprint available as arXiv:1205.5544v1, 2012.

\bibitem{rizell:lifting}
\bysame, \emph{Lifting pseudo-holomorphic polygons to the symplectisation of
  {$P \times R$} and applications}, Preprint available as arXiv:1305.1572,
  2013.

\bibitem{rudolph:qp-obstruction}
L.~Rudolph, \emph{Quasipositivity as an obstruction to sliceness}, Bull. Amer.
  Math. Soc. (N.S.) \textbf{29} (1993), no.~1, 51--59.

\bibitem{rudolph}
\bysame, \emph{An obstruction to sliceness via contact geometry and
  ``classical'' gauge theory}, Invent. Math \textbf{119} (1995), 155--163.

\bibitem{rudolph:survey}
\bysame, \emph{Knot theory of complex plane curves}, Handbook of knot theory,
  Elsevier B. V., Amsterdam, 2005, pp.~349--427.

\bibitem{rutherford:kauffman}
D.~Rutherford, \emph{Thurston-{B}ennequin number, {K}auffman polynomial, and
  ruling invariants of a {L}egendrian link: the {F}uchs conjecture and beyond},
  Int. Math. Res. Not. (2006), Art. ID 78591, 15.

\bibitem{rulings}
J.~Sabloff, \emph{Augmentations and rulings of {L}egendrian knots}, Int. Math.
  Res. Not. (2005), no.~19, 1157--1180.

\bibitem{stoimenow:pos}
A.~Stoimenow, \emph{On polynomials and surfaces of variously positive links},
  J. Eur. Math. Soc. (JEMS) \textbf{7} (2005), no.~4, 477--509.

\bibitem{tanaka:max-tb-pos}
T.~Tanaka, \emph{Maximal {B}ennequin numbers and {K}auffman polynomials of
  positive links}, Proc. Amer. Math. Soc. \textbf{127} (1999), no.~11,
  3427--3432.

\bibitem{yokota:pos-poly}
Y.~Yokota, \emph{Polynomial invariants of positive links}, Topology \textbf{31}
  (1992), no.~4, 805--811.

\end{thebibliography}

\end{document}